\documentclass[12pt]{article}
\usepackage[margin=1in]{geometry}
\usepackage{url}

\usepackage[T1]{fontenc}

\usepackage{amsmath}
\usepackage{amssymb}
\usepackage{amsthm}
\usepackage{mathrsfs}
\usepackage{verbatim}
\usepackage{graphicx}

\newtheorem{theorem}{Theorem}
\newtheorem{corollary}[theorem]{Corollary}
\newtheorem{lemma}[theorem]{Lemma}
\newtheorem{proposition}[theorem]{Proposition}

\theoremstyle{remark}

\newtheorem{definition}{Definition}

\newcommand{\Q}{\mathbb{Q}}
\newcommand{\N}{\mathbb{N}}
\newcommand{\Z}{\mathbb{Z}}
\newcommand*\colvec[1]{\begin{pmatrix}#1\end{pmatrix}}
\def\modd#1#2{#1\ \mbox{\rm (mod}\ #2\mbox{\rm )}}

\title{Critical exponents of infinite balanced words}
\author{
Narad Rampersad\\
University of Winnipeg (Math/Stats),\\
Winnipeg, MB, R3B 2E9, CANADA\\
\url{n.rampersad@uwinnipeg.ca}\bigskip\\
Jeffrey Shallit\\
School of Computer Science\\
University of Waterloo\\
Waterloo, ON, N2L 3G1, CANADA\\
\url{shallit@uwaterloo.ca}\bigskip\\
\'Elise Vandomme\\
Laboratoire de Combinatoire et d'informatique Math\'ematique (LaCIM)\\
Universit\'e du Qu\'ebec \`a Montr\'eal\\
Montr\'eal, QC, H3C 3P8, CANADA\\
\url{elise.vandomme@lacim.ca}}

\begin{document}
\maketitle

\begin{abstract}
Over an alphabet of size $3$ we construct an infinite balanced word with
critical exponent $2+\sqrt{2}/2$.  Over an alphabet of size $4$
we construct an infinite balanced word with critical exponent $(5+\sqrt{5})/4$.
Over larger alphabets, we give some candidates for balanced words (found
computationally) having small critical exponents.  We also explore a method for
proving these results using the automated theorem prover Walnut.
\end{abstract}

\section{Introduction}
A word $w$ (finite or infinite) is \emph{balanced} if, for any two
factors $u$ and $v$ of $w$ of the same length, the number of
occurrences of each alphabet symbol in $u$ and $v$ differ by at most
$1$.  Vuillon \cite{Vui03} gives a survey of some of the work done previously on
balanced words.  It is well-known that over a binary alphabet the class of
infinite aperiodic balanced words is exactly the class of Sturmian
words.  Sturmian words have been studied extensively (see the survey
by Berstel and S\'e\'ebold \cite{BS02}), and in
particular, much is known about the repetitions that occur in Sturmian
words.  The \emph{critical exponent} of an infinite word $w$ is the
supremum of the set of exponents of fractional powers appearing in
$w$.  Damanik and Lenz \cite{DL02} and Justin and Pirillo \cite{Jus01}
gave an exact formula for the critical exponent of a Sturmian word.
The Fibonacci word has critical exponent $(5+\sqrt{5})/2$ \cite{MP92},
and furthermore, by the formula previously mentioned
\cite{DL02, Jus01}, this is minimal over all Sturmian words.
In other words, over a binary alphabet, the least
critical exponent among all infinite balanced words is
$(5+\sqrt{5})/2$.  However, little is known about the critical exponents of
infinite balanced words over larger alphabets.  In
this paper we aim to construct infinite balanced words  over a given
alphabet having the smallest critical exponent  possible. Over an alphabet of
size $3$ we construct a balanced word with critical exponent
$2+\sqrt{2}/2$.  Over an alphabet of size $4$ we construct
a balanced word with critical exponent $(5+\sqrt{5})/4$.
Over larger alphabets, we give some candidates for balanced words (found
computationally) having small critical exponents.  We also explore a method for
proving these results using the automated theorem prover
Walnut \cite{Mou16}.

\section{Preliminaries}
We let $|w|$ denote the length of a finite word $w$, and if $a$ is a
letter of the alphabet, we let $|w|_a$ denote the number of occurrences of $a$ in
$w$.

\begin{definition}
A word $w$ (finite or infinite) over an alphabet $A$ is
\emph{balanced} if for every $a \in A$ and every pair $u, v$ of
factors of $w$ with $|u|=|v|$ we have
\[
|\,|u|_a - |v|_a\,| \leq 1.
\]
\end{definition}

Let $u$ be a finite word and write $u = u_0u_1 \cdots u_{n-1}$, where
the $u_i$ are letters.  A positive integer $p$ is a \emph{period}
of $u$ if $u_i = u_{i+p}$ for all $i$.  Let $e = |u|/p$ and let $z$ be
the prefix of $u$ of length $p$.  We say that $u$ has
\emph{exponent} $e$ and write $u = z^e$.  The word $z$ is called a
\emph{fractional root} of $u$.  Note that a word may have
multiple periods, and consequently, multiple exponents and fractional
roots.  The word $u$ is \emph{primitive} if the only integer exponent
of $u$ is $1$.  Let $w$ be a finite or infinite word.  The largest
$r\in \N$ (if it exists) such that $u^r$ is a factor of $w$ is the (integral)
\emph{index} of $u$ in $w$.

\begin{definition}
 The \emph{critical exponent} of an infinite word $w$ is 
\begin{align*}
E(w) &= \sup\{r\in\Q : \text{there is a finite, non-empty factor of $w$
       with exponent $r$}\}\\
 &=\inf\{r\in\Q : \text{there is no finite, non-empty factor of $w$
       with exponent $r$}\}.
\end{align*}
\end{definition}

The infinite words studied in this paper are constructed by modifying
Sturmian words.  The structure of such words are determined by a
parameter $\alpha$, which is an irrational real number between $0$ and
$1$, called the \emph{slope}, and more specifically, by the continued
fraction expansion $\alpha = [d_0, d_1, d_2, d_3, \ldots]$, where $d_i
\in \Z$ for $i \geq 0$.

\begin{definition}
The \emph{characteristic Sturmian word with
slope $\alpha$} (see \cite[Chapter~9]{AS03})
is the infinite word $c_\alpha$ obtained as the limit
of the sequence of \emph{standard words} $s_n$ defined by
\[
s_{0} = 0,\quad s_1 = 0^{d_1-1}1,\quad s_n = s_{n-1}^{d_n}s_{n-2},\quad n \geq 2.
\]
We write $c_\alpha[i]$ to denote the $i$-th letter of $c_\alpha$,
where we index starting from $1$.  For $n \geq 2$, we also define the
\emph{semi-standard words}
\[
s_{n,t} = s_{n-1}^ts_{n-2},
\]
where $t \in \Z$ and $1 \leq t < d_n$.
\end{definition}

We also make use of the \emph{convergents} of $\alpha$, namely
\[
\frac{p_n}{q_n} = [d_0, d_1, d_2, d_3, \ldots, d_n],
\]
where 
\begin{align*}
p_{-2} = 0,\quad p_{-1} = 1,\quad p_n = d_np_{n-1} + p_{n-2} \text{
  for } n \geq 0;\\
q_{-2} = 1,\quad q_{-1} = 0,\quad q_n = d_nq_{n-1} + q_{n-2} \text{
  for } n \geq 0.
\end{align*}
The convergents have the following approximation property:
\begin{equation}\label{eq:approx}
\left| \alpha - \frac{p_n}{q_n} \right| < \frac{1}{q_n q_{n+1}} < \frac{1}{q_n^2}.
\end{equation}

The following fact is classical:
\begin{equation}\label{eq:ratio}
\frac{q_{n+1}}{q_n} = [d_{n+1}, d_n, \ldots, d_1].
\end{equation}
Another classical result is that the continued fraction expansion for
$\alpha$ is ultimately periodic if and only if $\alpha$ is a quadratic
irrational.  We will need the reverse direction of the following
result as well.

\begin{theorem}[{\cite{LS93}}]\label{thm:lin_rec}
The sequences $(q_n)_{n \geq 0}$ and $(p_n)_{n \geq 0}$ satisfy a
linear recurrence with constant coefficients if and only if $\alpha$ is a
quadratic irrational.  For both sequences the linear recurrence is the
same and has the form
\begin{equation}\label{eq:lin_rec}
  q_{n+2s} - tq_{n+s} + (-1)^sq_n = 0,\quad n \geq r,
\end{equation}
for some $r,s,t$.
\end{theorem}

It is easy to see that $|s_n|=q_n$ for $n\geq 1$.  We need a more
precise result:

\begin{lemma}[{\cite[Lemma~9.1.9]{AS03}}]\label{lem:9.1.9}
For $n \geq 0$ we have $|s_n|_0=q_n-p_n$ and $|s_n|_1=p_n$.
\end{lemma}

Next we introduce the \emph{Ostrowski $\alpha$-numeration system}
\cite{Ost22} (see also \cite[Section~3.9]{AS03}).  Each
non-negative integer $N$ can be represented uniquely as
$b_jb_{j-1}\cdots b_0$, where
\[
N = \sum_{0 \leq i \leq j} b_iq_i,
\]
where the $b_i$ are integers satisfying:
\begin{enumerate}
\item $0 \leq b_0 < d_1$,
\item $0 \leq b_i \leq d_{i+1}$, for $i \geq 1$, and
\item for $i \geq 1$, if $b_i = d_{i+1}$, then $b_{i-1}=0$.
\end{enumerate}

The next two results give the connection between characteristic
Sturmian words and the Ostrowski numeration system.

\begin{theorem}[{\cite[Theorem~9.1.13]{AS03}}]\label{thm:9.1.13}
Let $N\geq 1$ be an integer with Ostrowski $\alpha$-representation
$b_jb_{j-1}\cdots b_0$.  Then the length-$N$ prefix of $c_\alpha$ is
equal to $s_j^{b_j}s_{j-1}^{b_{j-1}} \cdots s_0^{b_0}$.
\end{theorem}

\begin{theorem}[{\cite[Theorem~9.1.15]{AS03}}]\label{thm:9.1.15}
Let $N\geq 1$ be an integer with Ostrowski $\alpha$-representation
$b_jb_{j-1}\cdots b_0$.  Then $c_\alpha[N] = 1$ if and only if
$b_jb_{j-1}\cdots b_0$ ends with an odd number of $0$'s.
\end{theorem}

One characteristic Sturmian word is of particular significance.  Let $\phi =
(1+\sqrt{5})/2$.  The \emph{Fibonacci word} is the characteristic Sturmian
word
\[
c_\theta = 010010100100101001010010010100\cdots
\]
with slope $\theta := 1/\phi^2 = [0,2,\overline{1}\,]$.
We call the corresponding standard words the \emph{finite Fibonacci
  words}:
\[
f_0 = 0,\quad f_1 = 01,\quad f_2 = 010,\quad \ldots
\]
It is easy to verify that for each $i\geq 2$ the finite
Fibonacci word $f_i$ has length $F_{i+2}$ (the $(i+2)$-th Fibonacci
number) and has $F_{i+1}$ $0$'s and $F_{i}$ $1$'s.  It is also
well-known that the infinite Fibonacci word is fixed by the morphism
that maps $0 \to 01$ and $1 \to 0$.  Mignosi and Pirillo \cite{MP92}
showed that $E(c_\theta) = 2+\phi$.  The more general results of
Damanik and Lenz \cite{DL02} and Justin and Pirillo \cite{Jus01} show
that this is minimal over all Sturmian words.

\section{Constructing balanced words}

We need a characterization of recurrent aperiodic balanced words
due to Hubert \cite{Hub00} (Graham \cite{Gra73} gave an equivalent
characterization), which is based on the following notion:
a word $y$ has the \emph{constant gap} property if, for each letter $a$,
there is some number $d$ such that the distance between successive
occurrences of $a$ in $y$ is always $d$.

\begin{theorem}[\cite{Hub00}]\label{thm:hubert}
 A recurrent aperiodic word $x$ is balanced if and only if
$x$ is obtained from a Sturmian word $u$ over $\{0,1\}$ by the following
procedure: replace the $0$'s in $u$ by a periodic sequence $y$ with
constant gaps over some alphabet $A$ and replace the
$1$'s in $u$ by a periodic sequence $y'$ with constant gaps over
some alphabet $B$, disjoint from $A$.
\end{theorem}

Note that we shall only consider recurrent infinite words in this
paper, since a well-known result of Furstenberg \cite{Fur81} states
that for every infinite word $x$ there exists a uniformly recurrent
infinite word $x'$ such that every factor of $x'$ is a factor of $x$.

For $3 \leq k \leq 10$ we define an infinite word $x_k$ constructed
``\`a la Hubert'' from a Sturmian word $c_\alpha$, where we set
$\alpha$, $y$ and $y'$ according to Table~\ref{table:periodic_seqs}.

\begin{table}[h]
\centering
\begin{tabular}{cllll}
 $k$& $\alpha$ & c.f. & $y$ & $y'$ \\ \hline
 3  & $\sqrt{2} - 1$ & $[0, \overline{2}\,]$ & $(01)^\omega$ & $2^\omega$ \\
 4  & $1/\phi^2$ & $[0, 2, \overline{1}\,]$ & $(01)^\omega$ & $(23)^\omega$ \\
 5  &  $\sqrt{2} - 1$ & $[0, \overline{2}\,]$ & $(0102)^\omega$ & $(34)^\omega$ \\
 6  & $ (78 - 2\sqrt{6})/101$ & $[0,1,2,1,1,\overline{1,1,1,2}\,]$ & $0^\omega$ & $(123415321435)^\omega$ \\
 7  & $(63-\sqrt{10})/107$ & $[0,1,1,3,\overline{1,2,1}\,]$ & $(01)^\omega$ & $(234526432546)^\omega$ \\
 8  & $(23+\sqrt{2})/31$ & $[0,1,3,1,\overline{2}\,]$ & $(01)^\omega$ & $(234526732546237526432576)^\omega$ \\
 9  & $(23-\sqrt{2})/31$ & $[0,1,2,3,\overline{2}\,]$ & $(01)^\omega$ & $(234567284365274863254768)^\omega$ \\
 10 & $(109+\sqrt{13})/138$ & $[0,1,4,2,\overline{3}\,]$ & $(01)^\omega$ & $(234567284963254768294365274869)^\omega$ \\
\end{tabular}
\caption{$c_\alpha$ and constant gap  words $y$ and $y'$ for the construction of $x_k$}
\label{table:periodic_seqs}
\end{table}

We will show that
\[ E(x_3) = 2+\frac{\sqrt{2}}{2} \text{ and } E(x_4) = 1+\frac{\phi}{2}. \]
For $k \geq 5$, computer calculations suggest that
\[ E(x_k) = \frac{k-2}{k-3}. \]
Backtracking searches for $5 \leq k \leq 9$ show that this would be
the least possible critical exponent for balanced words over a
$k$-letter alphabet.  However, the backtracking algorithm used to establish
this is not the usual one, but instead involves backtracking over the
tree of \emph{standard pairs}.  This an (infinite) binary tree with
root $(0,1)$; each vertex $(u,v)$ has children $(u,uv)$ and $(vu,v)$
(see \cite[p.~254]{deL97} for more details).  Every finite balanced
binary word appears as a factor of either $u$ or $v$ for some node $(u,v)$
appearing in this tree.

The backtracking search proceeds as follows: We wish to prove that no
balanced word with critical exponent less than $B$ is possible over a
$k$-letter alphabet.  We start with the standard pair $(0,1)$ and use
breadth-first search.  In the queue are nodes of the tree yet to be
expanded.  We pop the queue and look at the longer of the two words,
$x$.  Let $L_A$ be the set of all primitive roots of words with the constant
gap property over the alphabet $A$\footnote{For more information on
  how to enumerate the set $L_A$, see the paper by Goulden et
  al.~\cite{GGRS17}.  Constant gap words over a $k$-letter
  alphabet are equivalent to \emph{exact covering systems} of size
  $k$.  In the paper by Goulden et al., they count a subset of these
consisting of the \emph{natural exact covering systems}; however, the
numbers of exact covering systems and natural exact covering systems
are equal up to size $12$.}.  For $i\leq k/2$, let $A_i=\{0,1,\ldots,i-1\}$ and
$B_i=\{i,i+1,\ldots,k-1\}$.  For each word $w \in L_{A_i}$ and $w' \in
L_{B_i}$, we substitute, changing, for each $j$, the $j$-th occurrence of $0$ in $x$
to the $j$-th letter of $w^\omega$ and the $j$-th occurrence of $1$ in
$x$ to the $j$-th letter of $(w')^\omega$. (We do not have to do the remaining ones
where $i > k/2$ because the standard pairs as enumerated contain both
the word $x$ and its complement.)  We then compute the critical
exponent of the resulting word.  If it is at least $B$, we discard this
node.  Otherwise we add the two children of the current node to the
end of the queue.  If the queue eventually becomes empty then we have
proved that no balanced word over a $k$-letter alphabet has critical
exponent less than $B$.

\section{Establishing the critical exponent}

\begin{proposition}\label{prop:k3}
The word $x_3$ has critical exponent
$$E(x_3) = 2+\frac{\sqrt{2}}{2} \approx 2.7071.$$
\end{proposition}

\begin{proof}
As stated in Table~\ref{table:periodic_seqs},
let $\alpha = \sqrt{2}-1 = [0, \overline{2}\,]$ and let $c_\alpha$ be the
characteristic Sturmian word with slope $\alpha$.  That is, the word
$c_\alpha$ is the infinite word obtained as the limit of the sequence
of standard words $s_k$ defined by
\[
s_0 = 0,\quad s_1 = s_01,\quad s_k = s_{k-1}^2s_{k-2},\quad k \geq 2
\]
Let $x_3$ be the word over $\{0,1,2\}$ obtained from $c_\alpha$ by
replacing the $0$'s in $c_\alpha$ by the periodic sequence
$(01)^\omega$ and by replacing the $1$'s with $2$'s.
We have $s_1 = 01$, $s_2=01010$, $s_3=(01010)^201$, etc., and
\[
c_\alpha = 010100101001010100101001010100101001010010\cdots
\]
and
\[
x_3 = 021201202102120210212012021201202102120\cdots.
\]

Let $(z')^e$ be a repetition of exponent $e \geq 2$ in $x_3$ ($e \in
\mathbb{Q}$).  Then, by applying the morphism that sends $\{0,1\}\to 0$ and
$2 \to 1$ to $x_3$, we see that
there is a corresponding repetition $z^e$ of the same length in $c_\alpha$.
Suppose that $z$ is primitive.  By \cite[Corollary~4.6.6]{Pel16} (originally due to Damanik and Lenz \cite{DL02})
$z$ is either a conjugate of one of the standard words $s_k$ 
defined above
or a conjugate of one of the semi-standard words
\[
s_{k,1} = s_{k-1}s_{k-2},\quad k \geq 2.
\]

Suppose that $z$ is a conjugate of a standard word.  Note that
$|s_k|_0$ is odd for every $k \geq 1$.  Hence $|z|_0$ is odd.  It follows that $z'z'$
cannot occur in $x_3$ (the alternations of $0$ and $1$ will not
``match up'' in the two copies of $z'$), and so there is no repetition
$(z')^e$ in $x_3$.

Now suppose that $z$ is a conjugate of a semi-standard word.  Then
$|z| = q_{k-2}+q_{k-1}$ for some $k \geq 2$.
From~\cite[Theorem 4(i)]{Jus01}, one finds that the longest
factor of $c_\alpha$ with this period has length
$2(q_{k-2}+q_{k-1})+q_{k-1}-2$.  It follows that
\begin{align*}
e &\leq \frac{2(q_{k-2}+q_{k-1})+q_{k-1}-2}{q_{k-2}+q_{k-1}}\\
&= 2 + \frac{q_{k-1}-2}{q_{k-2}+q_{k-1}}\\
&= 2 + \frac{q_{k-1}/q_{k-2}-2/q_{k-2}}{1+q_{k-1}/q_{k-2}}.
\end{align*}
Now by \eqref{eq:ratio} we have that $q_{k-1}/q_{k-2}$ converges to
$[2,\overline{2}\,] = \sqrt{2}+1$, and by \eqref{eq:approx} we have
$q_{k-1}/q_{k-2} < \sqrt{2}+1+1/q_{k-2}^2$.
Thus, we have
\[
e < 2 + \frac{\sqrt{2}+1+1/q_{k-2}^2-2/q_{k-2}}{\sqrt{2}+2-1/q_{k-2}^2}.
\]
The fraction on the right clearly tends to $(\sqrt{2}+1)/(\sqrt{2}+2) =
\sqrt{2}/2$ as $k\to \infty$, and is increasing for $k \geq 3$, so the
convergence is from below.  Thus $e < 2+\sqrt{2}/2$.

Indeed, for every $k \geq 2$, there are such factors with exponent
\[
2 + \frac{(q_{k-1}-2)/q_{k-2}}{1+q_{k-1}/q_{k-2}}
  \xrightarrow{k \to \infty} 2 + \frac{\sqrt{2}}{2},
\]
where the convergence is from below.  Note that
$|s_{k,1}|_0$ is even for every $k \geq 2$ and so every such repetition $z^e$ in
$c_\alpha$ gives rise to a repetition $(z')^e$ in $x_3$, since $|z|_0$ in this case is
even.

Finally, suppose that $z$ is not primitive.
By~\cite[Proposition~4.6.12]{Pel16}, the critical exponent of
$c_\alpha$ is $3+\sqrt{2}$, so clearly $z$ cannot have exponent $\geq
3$.  If $z$ is a square we have 
\[
e < \frac{3+\sqrt{2}}{2} < 2+\frac{\sqrt{2}}{2}.
\]
Thus $E(x_3) = 2+\frac{\sqrt{2}}{2}$.
\end{proof}

Next we show that the exponent in the previous result is the least
possible over $3$ letters.

\begin{proposition}\label{prop:k3_optimal}
Every balanced word $x$ over a $3$-letter alphabet has
critical exponent $$E(x) \geq 2+\frac{\sqrt{2}}{2} \approx 2.7071.$$
\end{proposition}

\begin{proof}
Let $\alpha = [0,d_1,d_2,d_3,\ldots]$ and let $c_\alpha$ be the
Sturmian word with slope $\alpha$.  We substitute the $0$'s in $c_\alpha$
with periodic word $y$ and the $1$'s with periodic word $y'$.  We
don't have much choice on $3$ letters, so let $y = (01)^\omega$ and
$y' = 2^\omega$.  Let $x$ be the resulting word and suppose that $x$
is cubefree.

Consider the standard words $s_0 = 0$, $s_1 = 0^{d_1-1}1$ and
$s_k = s_{k-1}^{d_k} s_{k-2}$ for $k\geq 2$.
Let $k\geq 2$.  We apply \cite[Theorem~4.6.5(iii)]{Pel16} (which is a
restatement of a result of \cite{DL02}), which
states that the index of the reversal of $s_k$ (denoted $s_k^R$) in
$c_\alpha$ is $d_{k+1}+2$.

The first observation is that the number of $0$'s in $s_k$ must be
odd.  If $|s_k|_0$ is even then $(s_k^R)^3$ occurs in $c_\alpha$ and
produces a cube in $x$, even after replacement of the $0$'s by
$(01)^\omega$.

Next we observe that $d_{k+2}$ must be even.  The number of $0$'s in
$s_{k+2}$ is $d_{k+2}|s_{k+1}|_0 + |s_{k}|_0$, and since $|s_k|_0$ and
$|s_{k+1}|_0$ are odd, we must have $d_{k+2}$ even.
Now if $d_{k+2} \geq 4$, then the index of $s_{k+1}^R$ is $\geq
6$, and $(s_{k+1}^R)^6$ will produce a cube in $x$ after replacement of the
$0$'s with $(01)^\omega$.

So $d_k = 2$ for $k\geq 4$; i.e., $\alpha = [0, d_1, d_2,
  d_3,\overline{2}\,]$.  As in Proposition~\ref{prop:k3}, we apply
\cite[Theorem 4(i)]{Jus01}, which states that $c_\alpha$ contains
repetitions whose fractional roots are conjugates of the semi-standard words
$s_{k-1}s_{k-2}$ (which have an even number of $0$'s) and whose
exponents are of the form $2 + (q_k-2)/(q_{k-1} + q_k)$.  Because the
fractional roots have an even number of $0$'s, after replacement of
the $0$'s by $(01)^\omega$, the corresponding factor of the word $x$
still has exponent $2 + (q_k-2)/(q_{k-1} + q_k)$.  A calculation
similar to the one in the proof of Proposition~\ref{prop:k3} shows
that as $k\to \infty$, this quantity again converges to $2 +
\sqrt{2}/2$ (regardless of the first few $d_i$'s).  The same argument
(counting parities of $1$'s this time) applies if $y = 2^\omega$ and
$y'=(01)^\omega$.
\end{proof}

In order to establish the critical exponent for $x_4$,
we first need a technical lemma concerning repetitions in the
Fibonacci word.

\begin{lemma}\label{lem:non-minimal_period}
Let $w$ be a factor of the Fibonacci word.  Write $w=x^f$, where $f
\in \mathbb{Q}$ and $|x|$ is the least period of $w$.  Suppose that
$w$ has another representation $w=z^e$, where $e \in \mathbb{Q}$,
$z$ is primitive, and $|x|<|z|$.  Then $e<1+\phi/2$.
\end{lemma}

\begin{proof}
First, recall that the critical exponent of the Fibonacci word is
$2+\phi$ \cite{MP92}.  Let $g=\lceil f-1 \rceil$ and write $x^f = x^gx'$ (so
$x'$ may possibly equal $x$).  We must have $e<2$,
since otherwise we would have $|w| \geq |z|+|x|$, and so by the
Fine--Wilf Theorem (see \cite[Theorem~1.5.6]{AS03}), $w$ would have period $\gcd(|x|,|z|)$,
contradicting the minimality of $x$ ($|z|$ cannot be a multiple of
$|x|$, since $z$ is primitive).  We thus write $w=zz'$, where $z'$ is
a non-empty prefix of $z$.  For the same reason we have $e<2$ we must
also have $|z'|<|x|$.  We now consider several cases.

\bigskip

\noindent Case~1: $g = 1$.  Then $|z'|<|x'|$ and $x'$ has $z'$
as both a prefix and a suffix.  It follows that $x'$ has period
$|x'|-|z'|$ and hence exponent
\[
\frac{|x'|}{|x'|-|z'|} < 2+\phi.
\]
This implies that
\[
|z'| < \left(\frac{1+\phi}{2+\phi}\right)|x'| <
\left(\frac{1+\phi}{2+\phi}\right)\cdot \frac12|w|.
\]
Thus 
\[
\frac{|z'|}{|w|} < \left(\frac{1+\phi}{2+\phi}\right)\cdot \frac12,
\]
and so
\[
e=\frac{|w|}{|z|} < \frac{1}{1-\left(\frac{1+\phi}{2+\phi}\right)\cdot \frac12}
\approx 1.5669 < 1+\frac{\phi}{2}.
\]

\bigskip

\noindent Case~2: $g = 2$.  If $|z'|<|x'|$, we apply the
argument of Case~1, so suppose $|z'| > |x'|$.

\bigskip

\noindent Subcase~2a: $|x'| \geq |x|/2$.  Then since $|z'|<|x|$, we have
\[
\frac{|z'|}{|w|} < \frac{|x|}{|xxx'|} \leq \frac25,
\]
and so
\[
e=\frac{|w|}{|z|} \leq \frac{1}{1-\frac25} = \frac53 < 1+\frac{\phi}{2}.
\]

\bigskip

\noindent Subcase~2b: $|x'| < |x|/2$.  Then $xx'$ has $z'$ as both a prefix and
a suffix.  It follows that $xx'$ has period $|xx'|-|z'|$ and hence
exponent
\[
\frac{|xx'|}{|xx'|-|z'|} < 2+\phi.
\]
This implies that
\[
|z'| < \left(\frac{1+\phi}{2+\phi}\right)|xx'|.
\]
Thus 
\[
\frac{|z'|}{|w|} < \frac{\left(\frac{1+\phi}{2+\phi}\right)|xx'|}{|xxx'|} <
\left(\frac{1+\phi}{2+\phi}\right)\cdot\frac35,
\]
where we have used the fact that $|x'| < |x|/2$ to obtain the last inequality.
Therefore,
\[
e=\frac{|w|}{|z|} < \frac{1}{1-\left(\frac{1+\phi}{2+\phi}\right)\cdot \frac35}
\approx 1.7673 < 1+\frac{\phi}{2}.
\]

\bigskip

\noindent Case~3: $g = 3$.  If $|z'|<|x'|$, we apply the
argument of Case~1.  If $|x'| < |z'|<|xx'|$, then, noting that $|xx'|<|w|/2$,
we again apply the argument of Case~1, but with $xx'$ in place of
$x'$.  The case $|z'| > |xx'|$ is impossible since we have $|z'|<|x|$.
\end{proof}

\begin{proposition}\label{prop:k4}
The word $x_4$ has critical exponent
  $E(x_4) = 1+\frac{\phi}{2} \approx 1.8090.$
\end{proposition}

\begin{proof}
Let $c_\theta$ be the Fibonacci word and let
$$x_4 = 021031201301203102130120130210310213012\cdots$$
be constructed from $c_\theta$ as
described above with $y = (01)^\omega$ and $y' = (23)^\omega$.

Consider a repetition of the form $(z')^e$ in $x_4$, where
$e\in\mathbb{Q}$ and $e>1$. If this repetition occurs at some position in
$x_4$, then there is a repetition $z^e$ of the same length at the
same position in $c_\theta$.  We will assume that $|(z')^{e-1}| \geq 3$: for
the case $|(z')^{e-1}| < 3$, one can verify with a little thought (or by
computer) that $e<1+\phi/2$.  We therefore have that $z^{e-1}$
contains at least one $0$ and one $1$, and so
$(z')^{e-1}$ contains at least one letter from $\{0,1\}$ and at
least one letter from $\{2,3\}$.  It follows that
\begin{equation}\label{eq:period_mod_n}
|z'|_{0} + |z'|_{1} \equiv \modd{0}{2}\text{ and }
|z'|_{2} + |z'|_{3} \equiv \modd{0}{2}.
\end{equation}

Saari \cite[Theorem~4.4]{Saa08} characterized the minimal fractional
roots of factors of Sturmian words.
In the case of the Fibonacci word, this characterization states that
the minimal fractional root of any
factor of the Fibonacci word is a conjugate of a finite Fibonacci word.

Thus, if $z$ is the minimal fractional root of the repetition $z^e$,
then $|z|$ equals $F_i$ (the $i$-th
Fibonacci number) for some $i$, and furthermore, we have
$|z|_0 = F_{i-1}$ and $|z|_1 = F_{i-2}$.
Equation~\eqref{eq:period_mod_n}
thus implies that
$$ 
 F_{i-1}=|z|_0=|z'|_{0} + |z'|_{1} \equiv \modd{0}{2}\text{ and }
 F_{i-2}=|z|_1=|z'|_{2} + |z'|_{3} \equiv \modd{0}{2},
$$
which is impossible.

Now suppose that $z$ is primitive but is not the minimal fractional
root of the repetition $z^e$.  Then by
Lemma~\ref{lem:non-minimal_period}, we have $e < 1+\phi/2$, as
required.

Finally, if $z$ is not primitive, then there exists a word $v$ such
that $z=v^2$ or $z=v^3$.  Thus either $z^e = v^{2e}$ or
$z^e = v^{3e}$.  Since $E(c_\theta)=2+\phi$, it follows that either
$2e<2+\phi$ or $3e<2+\phi$.  Both cases lead
to the inequality $e < 1 + \phi/2$, as required.

Finally, note that since $E(c_\theta)=2+\phi$, there is a sequence of
repetitions in $c_\theta$ with exponents $2+r_j/s_j$, with $r_j/s_j$
converging to $\phi$ from below.  Consider such a repetition and write it
as $V^{2+r_j/s_j} = Z^{1+r_j/(2s_j)}$.  Then $V$ is a conjugate of a
  finite Fibonacci word, and so for some $i$ we have
$$
 2F_{i-1}=2|V|_0=|Z|_0 \equiv \modd{0}{2}\text{ and }
 2F_{i-2}=2|V|_1=|Z|_1 \equiv \modd{0}{2},
$$
so there is a corresponding repetition $(Z')^{1+r_j/(2s_j)}$ in $x_4$.
  Thus $E(x_4)=1+\phi/2$.
\end{proof}

A backtracking computer search shows that there is no infinite
balanced word over a $4$-letter alphabet with critical exponent
$<1.8088$.  The minimal possible critical exponent among all infinite
balanced words over $4$ letters is thus between $1.8088$ and $1+\frac{\phi}{2} \approx 1.8090$.

Note that an important element of the previous proof was the result of
Lemma~\ref{lem:non-minimal_period} concerning non-minimal periods of
the Fibonacci word.  Little seems to be known on this topic, so we
conclude this section with the following result.

\begin{proposition}\label{prop:fib_period_differences}
  Let $w$ be a factor of a Sturmian word $s$ and let $\{p_1 < p_2 <
\cdots < p_r\}$ be the set of periods of $w$.  For $i \in
\{1,\ldots,r\}$ let $P_i$ be the prefix of $w$ of length $p_i$ and for
$i \geq 2$, write $P_i = P_{i-1}E_i$.  Then for $i \in
\{2,\ldots,r\}$ the word $E_i$ is a conjugate of either a standard
word $s_k$ or a semi-standard word $s_{k,t}$ for some $k,t$.
Furthermore, the lengths of the $E_i$ are non-increasing.
\end{proposition}

\begin{proof}
We need the following fact (see \cite{MRS98}):  If a word $w$ has periods
$p$ and $q$ with $q<p$, then the prefix of $w$ of length $|w|-q$ has
period $p-q$.  We also use the result of \cite[Theorem~4.4]{Saa08}: If
$w$ is a non-empty factor of $s$ then the minimal fractional root of $w$
is a conjugate of a standard or semi-standard word.

For $i \in \{2,\ldots,r\}$, consider the prefix of $w$ of length
$|w|-p_{i-1}$.  This prefix has period $p_i-p_{i-1}$ and corresponding
fractional root $E_i$.  We claim that this
is its minimal fractional root.  Suppose to the contrary that it has a smaller
fractional root of length $q$.  Then $w$ has period $p_{i-1}+q$, which is less than
$p_i$, which is a contradiction.  By \cite[Theorem~4.4]{Saa08}, we
have that $E_i$ is a conjugate of a standard word $s_k$ or a
semi-standard word $s_{k,t}$.

To show that the lengths of the $E_i$ are non-increasing, suppose to
the contrary that $|E_{i+1}|>|E_i|$.  Then the suffix of $w$ of length
$|w|-p_{i-1}$ has period $|E_i|$ and so does the suffix of $w$ of
length $|w|-p_i$.  As $w$ has period $p_i$, it therefore has period
$p_i + |E_i| < p_{i+1}$, which is a contradiction.
This completes the proof.
\end{proof}

\begin{corollary}
  Let $w$ be a factor of the Fibonacci word $c_\theta$ and let $\{p_1 < p_2 <
\cdots < p_r\}$ be the set of periods of $w$.  For $i \in
\{1,\ldots,r\}$ let $P_i$ be the prefix of $w$ of length $p_i$ and for
$i \geq 2$, write $P_i = P_{i-1}E_i$.  Then for $i \in
\{2,\ldots,r\}$ the word $E_i$ is a conjugate of a finite Fibonacci
word.
\end{corollary}

\section{A computational approach}\label{sec:walnut}
We can also attempt to establish the critical exponents of each $x_k$
using the Walnut theorem-proving software \cite{Mou16} and the methods
of Du, Mousavi, Schaeffer, and Shallit \cite{DMSS16, MSS16}.  The method is based
on two things.  The first is the fact that the Fibonacci word
$c_\theta$ is a \emph{Fibonacci-automatic sequence}: that is, the
terms of $c_\theta$ can be computed by a finite automaton that takes a
number $n$ written in the Fibonacci numeration system as input and
outputs the $n$-th term of $c_\theta$.  The second important element
to the method of Mousavi et al.\ is that the addition relation
$\{(x,y,z) \in \mathbb{N}^3 : x+y=z\}$ can be recognized by a finite
automaton that reads its input in the Fibonacci numeration system.

To extend this method to an arbitrary characteristic word $c_\alpha$,
we first need to show that the terms of $c_\alpha$ can be computed by
a finite automaton that takes the Ostrowski $\alpha$-representation
of $n$ as input.  This is immediate from Theorem~\ref{thm:9.1.15}
above.  The second thing we need is the recognizability of the
addition relation for the Ostrowski $\alpha$-numeration system.
Hieronymi and Terry \cite{HT17} showed that addition is indeed
recognizable in the case when $\alpha$ is a quadratic irrational.
We now show that when $\alpha$ is a quadratic irrational, the Hubert
construction applied to $c_\alpha$ results in a word that is automatic
for the Ostrowki $\alpha$-numeration system.

\begin{theorem}
Let $\alpha$ be a quadratic irrational and let $c_\alpha$ be the
characteristic Sturmian word with slope $\alpha$.
Let $x$ be any word obtained by replacing the $0$'s in
$c_\alpha$ with a periodic sequence $y$ and replacing the
$1$'s with a periodic sequence $y'$.  Then $x$ is Ostrowski
$\alpha$-automatic.
\end{theorem}

\begin{proof}
  Let $p$ and $p'$ be the periods of $y$ and $y'$ respectively.  We
  need to show that there is a deterministic finite automaton with
  output that takes the Ostrowski $\alpha$-representation of $n$ as
  input and outputs $x[n]$.  To compute $x[n]$ it suffices to be able
  to compute $c_\alpha[n]$ as well as the number of $0$'s modulo $p$
  and $1$'s modulo $p'$ in the length-$n$ prefix of $c_\alpha$.
  Let $b_j b_{j-1} \cdots b_0$ be the Ostrowski
  $\alpha$-representation of $n$.
  By Theorem~\ref{thm:9.1.15}, there is an automaton that computes
  $c_\alpha[n]$ from $b_j b_{j-1} \cdots b_0$.  By
  Lemma~\ref{lem:9.1.9} and Theorem~\ref{thm:9.1.13} the number
  of $0$'s in the length $n$ prefix of $c_\alpha$ is 
  \[
  b_s(q_s - p_s) + b_{s-1}(q_{s-1} - p_{s-1}) + \cdots + b_0(q_0 - p_0).
  \]
  Since $\alpha$ is a quadratic irrational, its continued fraction
  expansion is ultimately periodic.  By Theorem~\ref{thm:lin_rec}, the
  sequences $(q_i)_{i\geq 0}$ and $(p_i)_{i\geq 0}$ both satisfy
  the same linear recurrence relation, and hence so does
  $(q_i - p_i)_{i\geq 0}$.  It follows that $( (q_i - p_i) \bmod{p}
  )_{i\geq 0}$ is ultimately periodic.  Based on this ultimately periodic
  sequence, it is easy to construct an automaton to compute 
  \[
  \left( b_s(q_s - p_s) + b_{s-1}(q_{s-1} - p_{s-1}) + \cdots +
    b_0(q_0 - p_0) \right) \bmod{p}
  \]
  given $b_s b_{s-1} \cdots b_0$ as input.  A similar argument applies
  for computing the number of $1$'s in the length-$n$ prefix of
  $c_\alpha$.  This completes the proof.
\end{proof}

As an example, Figure~\ref{fig:x4_aut} shows the Fibonacci-base
automaton that generates the word $x_4$ (the labels on the states
indicate the output symbol for that state).

\begin{figure}[h]
\centering
\includegraphics[scale=0.7]{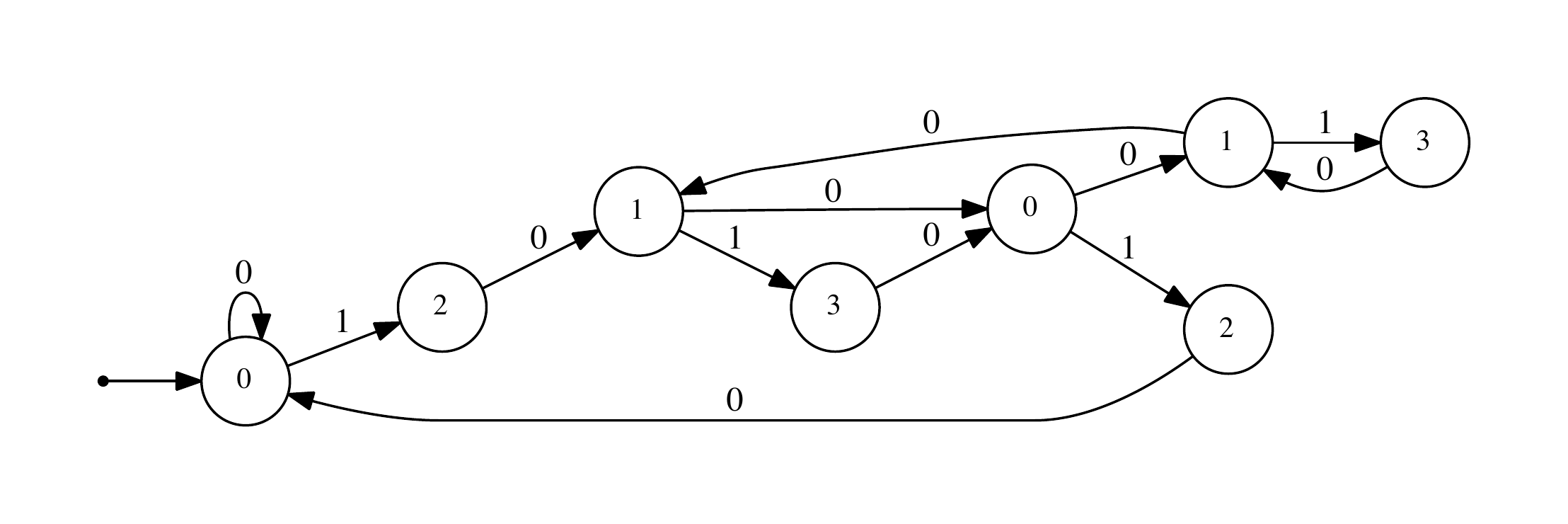}
\caption{Fibonacci-base automaton for $x_4$}
\label{fig:x4_aut}
\end{figure}

Let \texttt{X} denote the automaton in Figure~\ref{fig:x4_aut} in the
subsequent Walnut code.  First, we compute the periods $p$
such that a repetition with exponent $\geq 5/3$ and period $p$ occurs
in $x_4$:

\begin{verbatim}
eval periods_of_high_powers "?msd_fib Ei (p>=1) &
   (Aj (3*j <= 2*p) => X[i+j]=X[i+j+p])";
\end{verbatim}

The language accepted by the resulting automaton is $0^*1001000^*$;
i.e., representations of numbers of the form $F_n + F_{n-3} = 2F_{n-1}$.

\begin{verbatim}
reg pows msd_fib "0*1001000*";
\end{verbatim}

Next we compute pairs $(n,p)$ such that $x_4$ has a factor of length
$n+p$ with period $p$, and furthermore that factor cannot be extended
to a longer factor of length $n+p+1$ with the same period.

\begin{verbatim}
def maximal_reps "?msd_fib Ei (Aj (j<n) =>
   X[i+j]=X[i+j+p]) & (X[i+n]!=X[i+n+p])";
\end{verbatim}

We now compute pairs $(n,p)$ where $p$ has to be of the form $0^*1001000^*$
and $n+p$ is the longest length of any factor having that period.

\begin{verbatim}
eval highest_powers "?msd_fib (p >= 1) & $pows(p) & $maximal_reps(n,p) &
   (Am $maximal_reps(m,p) => m <= n)";
\end{verbatim}

The output of this last command is an automaton accepting pairs
$(n,p)$ having the form
\[
\colvec{0\\0}^*
\colvec{0\\1}\colvec{1\\0}\colvec{0\\0}\colvec{1\\1}\colvec{0\\0}
\left\lbrace \colvec{1\\0}\colvec{0\\0} \right\rbrace^* \colvec{0\\0}
\left\lbrace \epsilon, \colvec{1\\0} \right\rbrace.
\]
So when $p = 2F_{i-1}$ we see that $n=F_i-2$.
Thus, the maximal repetitions of ``large exponent'' in
$x_4$ have exponent of the form
$1+(F_i-2)/(2F_{i-1})$.  Using \eqref{eq:approx} and \eqref{eq:ratio}
we find that
\[
\frac{F_i-2}{2F_{i-1}} = \frac{F_i}{2F_{i-1}} - \frac{1}{F_{i-1}}
< \frac12 \left(\phi + \frac{1}{F_{i-1}^2}\right) - \frac{1}{F_{i-1}}
< \frac{\phi}{2}.
\]
Hence these exponents converge to $1+\phi/2$ from below, and we
conclude that $x_4$ has critical exponent $1+\phi/2$.

In principle, if the addition automaton of \cite{HT17} were
implemented in Walnut, it would be possible to carry out a similar
proof for $x_5$, $x_6$, etc., provided the resulting computation was
feasible.

\section{Future work}
The obvious open problem is to establish the claimed critical
exponents for $x_k$ for $5 \leq k \leq 10$, and more generally, show that for
$k \geq 5$, the least critical exponent for an infinite balanced word
over a $k$-letter alphabet is $(k-2)/(k-3)$.

\end{document}